\documentclass[12pt]{amsart} 

\usepackage{amsmath,amsthm}     
\usepackage{graphicx}     
\usepackage{hyperref} 
\usepackage{url}
\usepackage{amsfonts}

\theoremstyle{theorem}
\newtheorem{theorem}{Theorem}
\newtheorem{proposition}{Proposition}

\theoremstyle{definition}

\newtheorem*{remark}{Remark}
\newtheorem{question}{Question}
\newtheorem{example}{Example}

\date{}
\author{Ehssan Khanmohammadi}
\email{ehssan@pm.me}
\author{Omid Khanmohamadi}
\email{omidmath@pm.me}
\title[]{From Uniform Boundedness to the Boundary Between Convergence and Divergence}
\begin{document}
\maketitle
\section{Introduction}
Three of the fundamental ideas Stefan Banach introduced in functional analysis
together lead to his discovery of three fundamental results \cite[\S VI.84]{Kaluza1996,PCM2008}. The ideas were
\emph{abstract points} (functions as points, leading to operators and function
spaces), \emph{abstract sizes} (norms of functions, leading to distances
between functions), and \emph{abstract limits} (limits of sequences of
functions, leading to completeness of function spaces). The results were the
\emph{uniform boundedness principle}, the \emph{open mapping theorem}, and the
\emph{closed graph theorem}, which are all interrelated, in the sense that in
complete normed vector spaces (known as Banach spaces), Baire's category
theorem leads to several equivalences between qualitative properties (e.g.,
finiteness, surjectivity, regularity) and quantitative properties (e.g.,
estimates) of continuous (or equivalently bounded) linear operators \cite[\S 1.7]{Tao2010}.
One of these equivalences is captured by the uniform boundedness principle,
also known as the Banach--Steinhaus theorem. In this article we introduce a dual
of the uniform boundedness principle which does \emph{not} require \emph{completeness} and gives an
indirect means for testing the boundedness of a set. The dual principle, although known to the analyst and despite its applications in establishing results such as Hellinger--Toeplitz theorem, is often
missing from elementary treatments of functional analysis. In Example~\ref{Ex:counterexample} we indicate a connection between the dual principle and a question in spirit of du Bois-Reymond regarding the boundary between convergence and divergence of sequences. This example is intended to illustrate why the statement of the principle is natural and clarify what the principle claims and what it does not.
\section{Unbounded sets in normed spaces}
We begin with a proposition of linear algebraic flavor about the relation
between unbounded subsets of a normed space and the linear functionals on that
space. Below we shall assume that all vector spaces are over the field $\mathbb{R}$, and all linear maps between them are real-linear, although our results carry over easily to the field of complex numbers.

\begin{proposition}\label{P:Unbounded-Arbitrary Functionals} Let $S$ be an
  unbounded subset of a normed vector space $X$. Then there exists a linear
  functional $\phi\colon X\to \mathbb{R}$ whose restriction to $S$ has an
  unbounded image in $\mathbb{R}$. 
\end{proposition}
We include a simple proof here for the sake of completeness.
\begin{proof}
First assume $X$ is finite dimensional and let $n=\dim X$. Then since all norms on any finite-dimensional normed space are equivalent, we may (and we will) assume that the norm of $X$ is induced by an inner product. Fix an orthonormal basis $\{e_1, \dots, e_n\}$ with respect to this inner product for $X$  and for $i=1, \dots, n$ let $\operatorname{Proj}_{e_i}\colon X\to \mathbb{R}$ denote the scalar projection onto the $i$-th coordinate:
	\[
	\operatorname{Proj}_{e_i}(c_1e_1+\dots+c_ne_n)=c_i \quad \text{ for any } c_1, \dots, c_n\in \mathbb{R}.
	\]
We claim that for some value of $i$, the restriction of $\operatorname{Proj}_{e_i}$ to $S$ has an unbounded image. Indeed, if we had
  \[
    \sup_{x\in \operatorname{Proj}_{e_i} S}|x|\le M_i 
    \quad \text{with $M_i\ge 0$ for all } i=1, \dots, n,
  \]
  then it would follow that $\sup_{s\in S}\|s\|\le \sqrt{\sum M_i^2}<\infty $ contrary to
  the unboundedness of $S$.
  
  The same argument proves the proposition if $S$ (or an unbounded subset of
  $S$) is contained in a finite-dimensional subspace of $X$, or equivalently,
  if $\dim \operatorname{Span} S<\infty$.

  So we will assume that $S$ is not contained in any finite-dimensional subspace
  of $X$; then proceed to find an infinite linearly independent subset $\{b_1, b_2, \dots\}$
  of $S$ and extend $\{b_1, b_2, \dots\}$ to a possibly uncountable Hamel basis
  $B$ of $X$. Now define a function $\phi$ on $B$ by
  \[
    \phi(b)=
    \begin{cases}
      k & \text{ if } b=b_k,\\
      0 & \text{ otherwise},
    \end{cases}
    \qquad
    \text{ for }
    b\in B
  \]
and extend $\phi$ linearly to a functional, also denoted by $\phi$, on the entire space $X$. By construction, the restriction of $\phi$ to $\{b_1, b_2, \dots\}\subset S$ has an unbounded image and this completes the proof.
\end{proof}
The restriction of the functional $\phi$ to $S$ in the above proof had an unbounded image
as we required. However, the functional itself might also be ``unbounded'' or
discontinuous, an unwelcome phenomenon in analysis. Therefore, we can ask
whether unboundedness of $S$ can be captured by a \emph{continuous} linear
functional. Although it may not be a priori clear, this question is closely
related to the famous uniform boundedness principle in analysis. We explore
this relation in the next two sections. As we shall see, a central role in this regard is played by the notion of the \emph{operator norm}, denoted $\|T\|_{\text{op}}$, of a linear map of normed spaces $T\colon X\to Y$. We say that $T$ is \emph{bounded} when the operator norm defined by
	\[
	\|T\|_{\text{op}}=\sup_{\|x\|_X\le 1} \|Tx\|_Y
	\]
is finite. In words, $T$ is said to be bounded (as a function) if the image of the unit ball under $T$ is bounded (as a set). A simple observation that shows the importance of this definition in analysis is that boundedness and continuity are equivalent properties for linear maps of normed spaces.
\section{Uniform boundedness principle}
Before introducing its dual, let us first give a ``quantitative'' version of the uniform boundedness
principle itself. 
\begin{theorem}[Uniform boundedness principle]
  \label{thm:uniform-boundedness-principle} Let $X$ be a Banach space
  and let $Y$ be a normed space. Consider a family $F$ of bounded
  linear operators $T\colon X\to Y$. If $F$ is pointwise
  bounded, then it is uniformly bounded. 
  
  In fact, if $F$ is not uniformly bounded, then there exists a point $x\in X$ and a sequence $(T_n)$ of operators in $F$ satisfying $\|T_{n+1}\|_{\text{op}}> \|T_n\|_{\text{op}}$, $\|T_{n+1} x\|_Y> \|T_n x\|_Y$ for all $n$, and $\|T_n x\|_{\text{op}}\to \infty$.
\end{theorem}
\begin{proof}
Suppose $F$ is not uniformly bounded. Then we can find a sequence $(T_n)$ of nonzero operators in $F$ such that $\|T_{n+1}\|_{\text{op}}\ge 4^{2n+1}\|T_n\|_{\text{op}}$. Choose unit vectors $x_n$ such that $\|T_n x_n\|_Y\ge \frac{1}{2}\|T_n\|_{\text{op}}$.

For any $a, b\in X$, by the triangle inequality, at least one of the two inequalities $\|a+b\|_X\ge \|b\|_X$ and $\|a-b\|_X\ge \|b\|_X$ must hold. Thus, we may define a vector $x$ by $x=\sum_{k=1}^\infty \sigma(k) 4^{-k}x_k$ where $\sigma(k)$ takes its values from $\{\pm 1\}$ and it is defined recursively so that
	\[
	\left\|\sum_{k=1}^n  \sigma(k)4^{-k}T_nx_k\right\|_Y\ge \|4^{-k}T_n x_n\|_Y\ge \frac{1}{2}4^{-n}\|T_n\|_{\text{op}}.
	\]
Note that the series defining $x$ is absolutely convergent and hence convergent by completeness of $X$. The triangle inequality then implies that
	\begin{align}\label{eq:norm-comparison}
	\|T_n x\|_Y&\ge \notag
	\left\|\sum_{k=1}^n  \sigma(k)4^{-k}T_nx_k\right\|_Y
	-
	\left\|\sum_{k=n+1}^\infty  \sigma(k)4^{-k}T_nx_k\right\|_Y\\ 
	&\ge
	\frac{1}{2}4^{-n}\|T_n\|_{\text{op}}-
	\frac{1}{3}4^{-n}\|T_n\|_{\text{op}}=
	\frac{1}{6}4^{-n}\|T_n\|_{\text{op}}
	\end{align}
Since $\|x\|_X\le \frac{1}{3}$, we have $\|T_n x\|_Y\le \frac{1}{3}\|T_n\|_{\text{op}}$ and hence \eqref{eq:norm-comparison} yields
	\begin{align*}
	\|T_{n+1} x\|_Y&\ge \frac{1}{6}4^{-n-1}\|T_{n+1}\|_{\text{op}}\\
	&\ge \frac{1}{6}4^{-n-1}4^{2n+1}\|T_n\|_{\text{op}}=\frac{1}{6}4^n\|T_n\|_{\text{op}}\\
	&> \frac{1}{3}\|T_n\|_{\text{op}}\ge \|T_n x\|_Y
	\end{align*}
as desired.
\end{proof}
\begin{remark}
  Over the years, there have been numerous proofs of the uniform boundedness
  principle. These proofs may be
  categorized into those which use Baire's category theorem
  (``non-elementary'' proofs) and those which don't (``elementary'' proofs).
  Out of the ``elementary'' proofs the ``simple'' ones are of special
  interest; they usually make use of a ``gliding hump argument,'' such as the
  ones given by \cite[p.~51]{Carothers2005} or \cite{Sokal2011}.
  \cite{Carothers2005} reports that the original proof of the principle by
  Steinhaus and his protege Banach must have been an elementary proof of this
  kind, but apparently it was lost during the war. The non-elementary proof
  that survived was suggested as an alternative proof by Saks who refereed
  their paper! Some other elementary proofs (such as the one given by
  \cite[p.~63]{RieszNagy1990}) make use of a ``nested ball'' argument,
  similar to the argument used in the proof of Baire's category theorem. The advantage of the proof given above is its ``constructive'' nature (as opposed to most other proofs that are proofs by contradiction) which allows us to give a quantitative version of the uniform boundedness principle that we shall use in proving Theorem~\ref{thm:uniform-boundedness-dual}.
\end{remark}

\begin{subsection}{Norms in codomain and its dual}
Because the dual involves the codomain $Y$, let us say a few words about $Y$
in the uniform boundedness principle, which is merely a normed space. One of
the easy consequences of the Hahn--Banach theorem is the duality between the
definitions of the norm in $Y$ and in its dual $Y^*$ consisting of \emph{bounded} linear maps $y^*\colon Y\to \mathbb{R}$. More
precisely, for any $y^*\in Y^*$,
\[
\|y^*\|_{\text{op}}=\sup_{\|y\|\le 1}|y^*(y)|,
\]
and for any $y\in Y$,
\[
\|y\|=\sup_{\|y^*\|_{\text{op}}\le 1}|y^*(y)|,
\]
where in the second equality the supremum is attained.
\end{subsection}

\begin{remark} The most basic examples of normed spaces are, of course, the scalar
fields $\mathbb{R}$ and $\mathbb{C}$. It is perhaps interesting to
note that the uniform boundedness
principle for linear functionals (i.e., in the special case that the codomain is a scalar field) implies
the same theorem for all linear operators using the above remark about the
computation of norms. To see this, let $F$ be a family of bounded operators $T\colon X\to Y$ between normed spaces with $X$ complete. Suppose $F$ is pointwise bounded so that $\|T x\|\le M_x$ for some $M_x\ge 0$ depending on each $x$ in the unit ball of $X$. Then for each $y^*$ in the unit ball of $Y^*$, the functional $y^*\circ T$ is bounded and $|(y^*\circ T) (x)|\le M_x$. Thus by the uniform boundedness principle for functionals applied to the family $\{y^*\circ T\mid T\in F, y^*\in Y^* \text{ with } \|y^*\|_{\text{op}}\le 1\}$ we conclude that $|(y^*\circ T) (x)|\le M$ for some $M\ge 0$ independent of $x$. Taking the supremum over $y^*$, we obtain $\|T x\|\le M$, as claimed.
\end{remark}
\section{A dual for the uniform boundedness principle}
The appearance of the Hahn--Banach theorem, which is applicable to
general (i.e., not necessarily complete) normed spaces, in the last
section is not completely accidental. It suggests the idea that the
uniform boundedness principle might have some applications in the
context of general normed spaces as well. Our Hahn--Banach argument
proves, in particular, the following theorem, which is where Hahn
(1879--1934), Banach (1892--1945), and Steinhaus (1887--1972) meet,
posthumously! 
\begin{theorem}[Dual for the uniform boundedness principle]
  \label{thm:uniform-boundedness-dual} Let $S$ be a subset of a normed space $X$. If $\phi (S)$ is bounded for each $\phi\in X^*$,
  then $S$ is bounded.
  
  In fact, if $S$ is unbounded, then there exists $\phi\in X^*$ and a sequence $(s_n)$ in $S$ satisfying $\|s_{n+1}\|_X> \|s_n\|_X$, $|\phi(s_{n+1})|> |\phi(s_n)|$ for all $n$, and $|\phi(s_n)|\to \infty$.
\end{theorem}
Theorem~\ref{thm:uniform-boundedness-dual} can be thought of as a dual
for the uniform boundedness principle, since the boundedness of
$\phi(S)$ can be rephrased as the finiteness of
$\sup_{s\in S}|\phi(s)|$. Note that this theorem gives an
affirmative answer to the question that we raised after Proposition~\ref{P:Unbounded-Arbitrary Functionals}.
\begin{proof}
Since $\phi(S)$ is bounded, for $\phi\in X^*$,
  \[
  \sup_{s\in S}|s^{**}(\phi)|
  =
  \sup_{s\in S}|\phi(s)|
  < \infty.
  \]
This shows that the hypotheses of the uniform
  boundedness principle are satisfied for $F=\{s^{**}\colon X^*\to \mathbb{R}\mid s\in S\}$,
  thanks to the fact that the dual of any normed space is complete. Therefore, by the uniform
  boundedness principle and the fact that the map
  $x\mapsto x^{**}$ is an isometry from $X$ into the Banach space $X^{**}$,
  \[
  \sup_{s\in S}\|s\|_X
  =
  \sup_{s\in S}\|s^{**}\|_{X^{**}}
  <\infty,
  \]
  as desired. The last assertion follows from the second part of Theorem~\ref{thm:uniform-boundedness-principle}.
\end{proof}
Let us finish with a question about a possible strengthening of Theorem~\ref{thm:uniform-boundedness-dual} that we shall pick up in the next section.
\begin{question}\label{Q:beyond-subsequences} Suppose $S=\{s_1, s_2, \dots\}$ is a subset of a normed space $X$ such that $\|s_{n+1}\|_X>\|s_n\|_X$ for each $n$, and $\|s_n\|_X\to \infty$. Can we necessarily find a functional $\phi\in X^*$ satisfying $|\phi(s_{n+1})|>|\phi(s_n)|$ for each $n$, and $|\phi(s_n)|\to \infty$?
\end{question}
\section{Boundary between convergence and divergence}
We begin with a question---concerning the boundary between convergence and divergence of series---that first appeared in the work of Abel~\cite{Abel1828}, Dini\cite{Dini1867}, and du Bois-Reymond~\cite{Bois1873}.
\begin{question}\label{Q:Abel-du-Bois-Reymond} Suppose $\sum_{n=1}^\infty x_n$ is a convergent series with positive terms. Does there exist a sequence $(y_n)$ such that $y_n\to \infty$ and $\sum_{n=1}^\infty x_ny_n<\infty$? Similarly, suppose $\sum_{n=1}^\infty x_n$ is a divergent series with positive terms. Does there exist a sequence $(y_n)$ such that $y_n\to 0$ and $\sum_{n=1}^\infty x_ny_n=\infty$?
\end{question}
The answer to both of these, as it is well-known, is affirmative \cite{Knopp1928, Ash1997}. That is to say, there is neither a fastest convergent series nor a slowest divergent series. One can of course make analogous claims about sequences and, for instance, easily show that there is no slowest divergent sequence. Generalizing this, we pose a more restrictive question about the boundary between convergence and divergence of sequences.
\begin{question}\label{Q:main-question} Suppose $(x_n)$ is a sequence of numbers diverging to infinity. Does there exist a sequence $(y_n)$ such that $\sum_{n=1}^\infty y_n<\infty$ and $x_ny_n\to \infty$? What if we require $(y_n)\in \ell^p$?
\end{question}
Let us provide a quick comparison of the claims made in Questions~\ref{Q:Abel-du-Bois-Reymond} and \ref{Q:main-question}.
	\begin{center}\renewcommand{\arraystretch}{1.4}
	\begin{tabular}{|l|c|c|c|}
	\hline 
	 & Given & Wanted Convergence & Wanted Divergence\\ 
	\hline 
	Q\ref{Q:Abel-du-Bois-Reymond} & $\sum_{n=1}^\infty x_n=\infty$ & $y_n\to 0$& $\sum_{n=1}^\infty x_ny_n=\infty$ \\ 
	\hline 
	Q\ref{Q:main-question} & $x_n\to\infty$ & $\sum_{n=1}^\infty y_n<\infty$ & $x_ny_n\to \infty$ \\ 
	\hline 
	\end{tabular} 
	\end{center}
\begin{example}\label{Ex:counterexample} Let $x_n=\sqrt{n}$ for $n=1, 2, \dots$. Now we ask whether there exists a sequence $(y_n)\in \ell^2$ such that $x_ny_n\to \infty$ as $n\to \infty$. What makes the sequence $(x_n)$ worth studying in this context is the fact that $(\frac{1}{\sqrt{n^{1+\epsilon}}})\in \ell^2$ for all $\epsilon>0$, and $(\frac{1}{\sqrt{n^{1-\epsilon}}})\not\in \ell^2$ for all $\epsilon\ge 0$. Before answering this question, we indicate its connection with Theorem~\ref{thm:uniform-boundedness-dual}, the dual for the uniform boundedness principle. 

Let $(x_n)$ be a sequence of numbers such that $|x_{n+1}|> |x_n|$ for all $n$ and $|x_n|\to \infty$ and let $\{e_1, e_2, \dots\}$ be the standard orthonormal basis for $\ell^2$. Define a set $S$ by $S=\{x_1 e_1, x_2e_2, \dots\}$. Then $S$ is an unbounded subset of $\ell^2$ and hence, by Theorem~\ref{thm:uniform-boundedness-dual}, we can find $\phi\in (\ell^2)^*$ and a sequence $(s_n)$ in $S$ satisfying $\|s_{n+1}\|_2> \|s_n\|_2$, $|\phi(s_{n+1})|> |\phi(s_n)|$ for all $n$, and $|\phi(s_n)|\to \infty$. But since $|x_{n+1}|> |x_n|$, we must have $s_k=x_{n_k}e_{n_k}$ for a subsequence $(x_{n_k})$ of $(x_n)$. Thus, $|\phi(x_{n_{k+1}}e_{n_{k+1}})|> |\phi(x_{n_k}e_{n_k})|$ for all $n$, and $|\phi(x_{n_k}e_{n_k})|\to \infty$. To reveal the connection between the dual for the uniform boundedness principle and Question~\ref{Q:main-question}, we use the Riesz representation theorem which establishes the existence of a sequence $y=(y_n)\in \ell^2$ such that 
	\[
	\phi(x)=\langle x, y\rangle\qquad \text{ for all } x\in \ell^2.
	\]
Thus, for each $k$ we find $\phi(x_{n_k}e_{n_k})=\langle x_{n_k} e_{n_k}, y\rangle=x_{n_k}\langle  e_{n_k}, y\rangle=x_{n_k}y_{n_k}$. 

In conclusion, Theorem~\ref{thm:uniform-boundedness-dual} implies the existence of a square-summable sequence $(y_k)\in \ell^2$ such that $|x_{n_{k+1}}y_{k+1}|> |x_{n_k}y_k|$ and $|x_{n_k}y_k|\to \infty$ as $k\to \infty$ for a \emph{subsequence} $(x_{n_k})$ of $(x_n)$. But this statement is rather obvious! For instance, in the case of $(x_n)=(\sqrt{n})$ the \emph{subsequence} $(x_{n_k})$ defined by $x_{n_k}=\sqrt{k^4}=k^2$ and the sequence $(y_k)=(\frac{1}{k})$ have the desired properties.

This cannot be done, however, if we don't allow the passage to subsequences as in Questions~\ref{Q:beyond-subsequences} and \ref{Q:main-question}. To see this, we return to $(x_n)=(\sqrt{n})$ and let $(y_n)$ be any sequence such that $x_ny_n\to \infty$. Then $x_n^2y_n^2=ny_n^2\to\infty$. But since the harmonic series $\sum_{n=1}^\infty n^{-1}$ is divergent, and we are assuming that $y_n^2/n^{-1}\to \infty$, the limit comparison theorem for series implies that $\sum_{n=1}^\infty y_n^2$ must also be divergent, i.e., $(y_n)\not \in \ell^2$.
\end{example}
Questions of this nature arise in Fourier analysis and provide a means for measuring regularity of functions. For instance, for a periodic function $f\in L^2(\mathbb{T})$, we have $\widehat{f}\in \ell^2(\mathbb{Z})$, and if $f\in C^k(\mathbb{T}), k\ge 0$, then $\widehat{f}\in o(n^{-k})$, where $\widehat{f}(n)$ is the $n$th Fourier coefficient of $f$ given by $\widehat{f}(n)=\int_0^1 f(x)e^{-2\pi i nx}\, dx$ for each $n\in \mathbb{Z}$. See the Riesz--Fischer Theorem and the Riemann--Lebesgue Lemma~\cite{Tao2010}. Thus, ``the smoother the function, the faster the decrease of its Fourier coefficients.'' Indeed, a standard application of integration by parts shows that if $f\in C^k(\mathbb{T})$, then $(n^k\widehat{f}(n))\in \ell^2(\mathbb{Z})$. It would be nice if the converse were true, but it is false. It turns out that a weaker form of the converse is true, but we shall not state it here. Instead, we invite the interested reader to explore how this set of ideas leads to the definition of a Sobolev space.

\end{document}